\documentclass[11pt]{article}

\usepackage[margin=1in]{geometry}
\usepackage{amsmath,amssymb,amsthm,mathtools}
\usepackage{microtype}
\usepackage[hidelinks]{hyperref}
\usepackage{xcolor}

\newtheorem{theorem}{Theorem}
\newtheorem{lemma}[theorem]{Lemma}
\newtheorem{proposition}[theorem]{Proposition}
\newtheorem{corollary}[theorem]{Corollary}
\newtheorem{definition}[theorem]{Definition}
\newtheorem{question}[theorem]{Question}

\theoremstyle{definition}

\newcommand{\eps}{\varepsilon}

\title{Forcing monochromatic subdivisions}
\author{Gabriel Collado\footnote{Department of Mathematics, Statistics and Computer Science, University of Illinois, Chicago, IL 60607. Email: gcoll8@uic.edu.} \and 
Dhruv Mubayi\footnote{Department of Mathematics, Statistics and Computer Science, University of Illinois, Chicago, IL 60607. Email: mubayi@uic.edu. Research partially supported by NSF Awards DMS-2153576 and DMS-2552740
and a Simons Fellowship.}}
\date{}
\begin{document}

\maketitle

\begin{abstract}
We prove that for every integers $d \ge 1$ and $s\ge2$ there exists an integer $D$, that depends only on $d$ and $s$, such that for every graph $P$ with maximum degree at most $ d$, there is a graph $H$ with maximum degree at most $D$ in which every $s$-coloring of $V(H)$ yields a monochromatic subdivision of $P$. 
\end{abstract}

%--------------------------------------------------------------------
\section{Introduction}
%--------------------------------------------------------------------

Ramsey theory studies the emergence of monochromatic substructures under arbitrary colorings. In its classical form, one considers edge-colorings of complete graphs and seeks monochromatic copies of a fixed graph. More precisely, the \textit{classical Ramsey number} of a graph $G$,
denoted $R(G;s)$, is defined as
\[
  R(G;s) = \min\bigl\{ |V(K_n)| : K_n \xrightarrow{s} G \bigr\},
\]
where $K_n \xrightarrow{s} G$ means that every coloring of $E(K_n)$ with $s$ colors produces a monochromatic copy of $G$. The existence of such $n$ for any $G$ and $s$ was established by Ramsey~\cite{Ramsey1930} and, independently, by Erd\H{o}s and Szekeres~\cite{ErdosSzekeres1935}.
When $H \xrightarrow{s} G$ we call $G$ the \textit{target graph} and $H$ a \textit{Ramsey host} for $G$. When $s = 2$ we omit the number of colors and write $R(G)$ for $ R(G;2)$ and $H \xrightarrow{} G $ for $ H \xrightarrow{2} G$.
 
One can ask the same question for any monotone graph parameter $\rho$. The
\textit{$\rho$-Ramsey number} of a graph $G$ is defined as
\[
  R_\rho(G;s) = \min\bigl\{ \rho(H) : H \xrightarrow{s} G \bigr\}.
\]
This framework captures several well-studied variants. The classical Ramsey number corresponds to
$\rho(H) = v(H)$, the number of vertices of $H$. Another natural choice is $\rho(H) = e(H)$,
the number of edges of $H$, which gives the \textit{size Ramsey number} $\hat{r}(G;s)$,
introduced by Erd\H{o}s, Faudree, Rousseau, and Schelp~\cite{ErdosFaudreeRousseauSchelp1978}. A third
natural choice is $\rho(H) = \Delta(H)$, the maximum degree of $H$, which gives the
\textit{degree Ramsey number} $R_\Delta(G;s)$, studied systematically by Kinnersley, Milans, and
West~\cite{KinnersleyMilansWest2012}.
 
We focus on the case $\rho = \Delta$. The central question in this context is the following:
 
\begin{question}[\cite{HornMilansRodl2014}]\label{qn:main}
  Is $R_\Delta(G;s)$ bounded by a function of $\Delta(G)$ and $s$ alone?
\end{question}

Based on this question, we give the following definition:

\begin{definition}
    A family $\mathcal{G}$ of graphs is said to be \textit{degree-Ramsey bounded}, or \textit{$R_\Delta$-bounded}, if for every $s\in \mathbb{Z}^+$ and for every $G \in \mathcal{G}$ the degree-Ramsey number $R_\Delta(G;s)$ is bounded by a function depending only on $\Delta(G)$ and $s$.
\end{definition}

Question~\ref{qn:main} asks if the family of all graphs is \textit{$R_\Delta$-bounded}.
 
Several families are known to be $R_\Delta$-bounded.
Trees are $R_\Delta$-bounded with the explicit bound $R_\Delta(T;s) \leq 2s(\Delta(T)-1)$, as observed by Jiang. Cycles are also $R_\Delta$-bounded: Jiang, Milans, and West~\cite{JiangMilansWest2013} proved that $R_\Delta(C_n; 2)$ is at most $96$ for even cycles and at most $3458$ for odd cycles. More generally, Horn, Milans, and R\"{o}dl~\cite{HornMilansRodl2014} proved that the family of closed blowups of trees is $R_\Delta$-bounded. Recall that the \textit{closed $k$-blowup} of a graph $G$, denoted $G[k]$, is obtained by replacing every vertex of $G$ with a clique of size $k$ and every edge with a complete bipartite graph $K_{k,k}$. This last result plus the fact that every graph of bounded treewidth and bounded maximum degree is a subgraph of a closed blowup of a bounded-degree tree imply that the family of graphs with bounded treewidth is $R_\Delta$-bounded, as stated in Corollary~\ref{corol:bndtreewidth}.

Recall that a \textit{tree decomposition} of a graph $G$ is a tree $\mathcal{T}$ together with a collection of \textit{bags} $\{B_t \subseteq V(G) : t \in V(\mathcal{T})\}$ satisfying the following three conditions: every vertex of $G$ belongs to some bag; for every edge $uv \in E(G)$, some bag contains both $u$ and $v$; and for every vertex $v \in V(G)$, the set of nodes $t$ with $v \in B_t$ induces a connected subtree of $\mathcal{T}$. The \textit{width} of a tree decomposition is the size of its largest bag minus one, and the \textit{treewidth} of $G$, denoted $\mathrm{tw}(G)$, is the minimum width over all tree decompositions of $G$.

The following lemma states the fact that every graph of bounded treewidth and bounded maximum degree is a subgraph of a closed blowup of a bounded-degree tree in a precise manner.
\begin{lemma}[\cite{boundedtreewidth}]\label{lem:bddtreewidth}
    Every graph $G$ with treewidth $\mathrm{tw}(G) = w$ and maximum degree $\Delta(G) =  d$ is a subgraph of the closed blowup $T [18wd]$ for some tree $T$ of maximum degree at most $18wd^2$.
\end{lemma}

Thus, the result of Horn, Milans, and R\"{o}dl~\cite{HornMilansRodl2014} yields the following corollary.
 
\begin{corollary}\label{corol:bndtreewidth}
  The family of graphs with bounded treewidth is $R_\Delta$-bounded.
  More precisely, for every integers $s \ge 2$,  $d \ge 1$ and $w \ge 0$ there exists an integer $D$, depending only on $s, d$ and $w$, such that for every graph $P$ with maximum degree at most $d$ and treewidth at most $w$, there is a graph $H$ with maximum degree at most $D$ such that every $s$-coloring of $E(H)$ yields a monochromatic copy of $P$.
\end{corollary}
One can consider a natural variation of Ramsey numbers by coloring the \textit{vertices} of the host graph rather than its edges. This is the direction we explore. We define the \textit{vertex-$\rho$-Ramsey number} of $G$ as
\[
  R^v_\rho(G;s) = \min\bigl\{ \rho(H) : H \xrightarrow{s}_{\mathrm{vtx}} G \bigr\},
\]
where $H \xrightarrow{s}_{\mathrm{vtx}} G$ means that every $s$-coloring of $V(H)$ contains a
monochromatic copy of $G$ as a subgraph in some color class. Analogously, we get the following definition:

\begin{definition}
    We say a family $\mathcal{G}$ of graphs is \textit{vertex-$R_\Delta$-bounded} if for every $s\in \mathbb{Z}^+$ and for every $G \in \mathcal{G}$ the vertex-degree-Ramsey number $R^v_\Delta(G;s)$ is bounded by a function depending only on $\Delta(G)$ and $s$.
\end{definition}

Vertex Ramsey theory has been studied since the work of Folkman~\cite{Folkman1970} and Ne\v{s}et\v{r}il, and R\"{o}dl~\cite{NesetrilRodl1976}, who established the existence of
vertex Ramsey hosts of bounded clique number.  The main open problem here is the natural analogue of Question~\ref{qn:main}.

\begin{question}\label{qn:vertex-main}
  Is $R^v_\Delta(G;s)$ bounded by a function of $\Delta(G)$ and $s$ alone?
\end{question}
 Recall that a \textit{subdivision} of a graph $G$ is any graph obtained by replacing each edge of $G$ with an internally vertex-disjoint path of length at least one.
While we are not able to answer Question~\ref{qn:vertex-main}, our main result gives a positive answer once the required monochromatic target structure is enlarged from $G$ itself to the collection of subdivisions of $G$. Our main result is the following:
 
\begin{theorem}\label{thm:main}
  For every integers $s \ge 2$ and  $d \ge 1$ there exists an integer $D$, that depends only on $d$ and $s$, such that for every graph $P$ with maximum degree at most $ d$, there is a graph $H$ with maximum degree at most $D$ such that every $s$-coloring of $V(H)$ yields a monochromatic subdivision of $P$.
\end{theorem}
 
Our proof shows that blowups of bounded-degree expanders already suffice to enforce the desired vertex Ramsey property. The argument combines the expansion properties of a bounded-degree expander with a minor-universality theorem of Krivelevich and Nendov~\cite{KN}
to force a monochromatic minor of $P$ inside some color class, and then uses a blowup construction to lift this minor to a subdivision of $P$.

We also establish vertex analogues of \cite{HornMilansRodl2014} and Corollary~\ref{corol:bndtreewidth}. Namely, we prove that the family of closed blowups of trees, and more generally the family of graphs with bounded treewidth, are both vertex-$R_\Delta$-bounded. These results and proofs can be found in Section~\ref{sec:treewidth}.

%--------------------------------------------------------------------
\section{Notation and Preliminaries}
%--------------------------------------------------------------------

Throughout, all graphs are finite and simple, and all logarithms are
natural. For a graph $F$ and a set $X\subseteq V(F)$, write
\[
N_F(X)=\{y\in V(F)\setminus X: xy\in E(F)\text{ for some }x\in X\}
\]
for the external neighborhood of $X$ in $F$.

A \emph{subdivision} of a graph $P$ is obtained by replacing every edge of
$P$ by a path so that the replacing paths are internally vertex-disjoint.
Equivalently, the original vertices of $P$ become branch vertices, and two
replacing paths may meet only at a common prescribed branch vertex. A graph
$P$ is a \emph{minor} of a graph $G$ if there are pairwise disjoint
nonempty sets $W_v\subseteq V(G)$, indexed by $v\in V(P)$, such that every
$G[W_v]$ is connected and, for each $uv\in E(P)$, there are vertices 
$x_u \in W_u$ and $x_v \in W_v$ such that $x_ux_v \in E(G)$. The sets $W_v$ form a \emph{minor model} of $P$ in $G$.

For a positive integer $s$, write $[s]=\{1,\ldots,s\}$, and let $B_s(G)$
denote the $s$-\emph{blowup} of $G$: each vertex $x\in V(G)$ is replaced
by an independent set
\[
C_x=\{(x,1),\ldots,(x,s)\},
\]
in $B_s(G)$ and, for each edge $xy\in E(G)$, all edges between $C_x$ and $C_y$ are
added in $B_s(G)$.

If $G$ is an $r$-regular graph with adjacency matrix $A(G)$ and the eigenvalues of $A(G)$ are $r=\lambda_1\ge\lambda_2\ge\cdots\ge\lambda_N$. Define
\[
\lambda(G):=\max_{2\le i\le N}|\lambda_i|.
\]
In other words, $\lambda(G)$ is the second largest eigenvalue in absolute value of $A(G)$.

\begin{definition}
Let $0<\alpha<1$ and $t\ge 1$. A graph $F$ on $q$ vertices is an
$(\alpha,t)$-\emph{expander} if
\[
|N_F(X)|\ge t|X|
\qquad\text{for every }X\subseteq V(F)
\text{ with }|X|\le \frac{\alpha q}{t}.
\]
A graph is $m$-\emph{minor-universal} if it contains every graph with at
most $m$ vertices and at most $m$ edges as a minor.
\end{definition}

%--------------------------------------------------------------------
\section{Tools}
%--------------------------------------------------------------------

We use the following minor-universality theorem of Krivelevich and Nenadov.

\begin{theorem}[Krivelevich--Nenadov {\cite{KN}}]
\label{thm:KN}
For every $0<\alpha<1$ there are constants $\xi>0$ and
$q_0,t_0\in\mathbb{N}$ such that the following holds. If $F$ is an
$(\alpha,t)$-expander on $q\ge q_0$ vertices and $t\ge t_0$, then $F$ is
$m$-minor-universal for
\[
m=\xi\frac{q\log t}{\log q}.
\]
\end{theorem}

We also use the existence of fixed-degree spectral expanders.
Friedman's second-eigenvalue theorem implies the following; see, for
example, Bordenave~\cite{Bordenave}.

\begin{proposition}
\label{prop:spectral-existence}
For every fixed integer $r\ge 3$ and every $\eps>0$, for all sufficiently
large integers $N$ with $rN$ even there exists an $N$-vertex $r$-regular
graph $G$ such that
\[
\lambda(G)\le 2\sqrt{r-1}+\eps.
\]
\end{proposition}

For completeness, we record the form of the Expander Mixing Lemma that
will be used.

\begin{lemma}\label{lem:EML}
Let $G$ be an $N$-vertex $r$-regular graph with $\lambda(G)\le\lambda$.
If $X,Z\subseteq V(G)$ are disjoint, then
\[
\left|e_G(X,Z)-\frac{r}{N}|X||Z|\right|
   \le \lambda\sqrt{|X||Z|},
\]
where $e_G(X,Z)$ is the number of edges with one endpoint in $X$ and the
other in $Z$.
\end{lemma}

%--------------------------------------------------------------------
\section{Proof of Theorem~\ref{thm:main}}
%--------------------------------------------------------------------

We first extract a small-set expander from every large subset of a
sufficiently strong spectral expander.

\begin{lemma}\label{lem:robust}
Fix $s\in\mathbb{Z}^+$. Let $T\ge 100$ and put
\[
\eta=\frac{1}{10s\sqrt{T}}.
\]
Let $G$ be an
$N$-vertex $r$-regular graph satisfying
$\lambda(G)\le\eta r$. Then every set $S\subseteq V(G)$ with
$|S|\ge N/s$ contains a set $U\subseteq S$ with $|U|>N/(3s)$ such that
$G[U]$ is a $(1/(10s),T)$-expander.
\end{lemma}

\begin{proof}
Set
\[
a=\frac{N}{25sT}
\qquad\text{and}\qquad
b=\frac{N}{10sT}.
\]
We first claim that
\begin{equation}\label{eq:medium-expansion}
|N_{G[S]}(X)|\ge T|X|
\end{equation}
for every $X\subseteq S$ satisfying $a\le |X|\le a+b$.

Suppose otherwise, and set
\[
Y=N_{G[S]}(X),
\qquad
Z=S\setminus(X\cup Y).
\]
There are no edges between $X$ and $Z$. Moreover,
\[
|X|\le a+b=\frac{7N}{50sT}
\qquad\text{and}\qquad
|Y|<T|X|\le\frac{7N}{50s}.
\]
Consequently,
\[
|Z|\ge \frac{N}{s}-\frac{7N}{50sT}-\frac{7N}{50s}
     \ge \frac{N}{3s},
\]
where the last inequality uses $T\ge100$. By
Lemma~\ref{lem:EML} and $e_G(X,Z)=0$,
\[
\frac{r}{N}|X||Z|
 \le \lambda(G)\sqrt{|X||Z|}.
\]
It follows that
\[
|X||Z|
 \le \frac{\lambda(G)^2}{r^2}N^2
 \le \eta^2N^2
 =\frac{N^2}{100s^2T}.
\]
Since $|Z|\ge N/(3s)$, this gives
\[
|X|\le\frac{3N}{100sT}<\frac{N}{25sT}=a,
\]
a contradiction. This proves \eqref{eq:medium-expansion}.

We now perform a pruning procedure. Start with $U_0=S$. As long as there
is a nonempty set $X_i\subseteq U_{i-1}$ such that
\[
|X_i|\le b
\qquad\text{and}\qquad
|N_{G[U_{i-1}]}(X_i)|<T|X_i|,
\]
delete $X_i$ and set $U_i=U_{i-1}\setminus X_i$.

We claim that the procedure deletes fewer than $a$ vertices in total.
Otherwise, let $k$ be the first index for which
\[
W=X_1\cup\cdots\cup X_k
\]
has size at least $a$. By the minimality of $k$,
\[
a\le |W|<a+b.
\]
Furthermore,
\[
N_{G[S]}(W)
 \subseteq \bigcup_{i=1}^k N_{G[U_{i-1}]}(X_i),
\]
and therefore
\[
|N_{G[S]}(W)|
 \le \sum_{i=1}^k |N_{G[U_{i-1}]}(X_i)|
 < T\sum_{i=1}^k|X_i|
 =T|W|.
\]
This contradicts \eqref{eq:medium-expansion}.

Let $U$ be the set remaining when the procedure terminates. Then
\[
|U|>|S|-a
 \ge\frac{N}{s}-\frac{N}{25sT}
 >\frac{N}{3s}.
\]
By the stopping rule,
\[
|N_{G[U]}(X)|\ge T|X|
\]
for every nonempty $X\subseteq U$ with $|X|\le b$; the same inequality is
trivial for $X=\varnothing$. Since
\[
\frac{|U|}{10sT}\le\frac{N}{10sT}=b,
\]
the graph $G[U]$ is a $(1/(10s),T)$-expander.
\end{proof}

Recall that $B_s(P)$ denotes the $s$-blowup of $P$. The next lemma shows how to obtain a subdivision of a graph $P$ from a minor of $P$ by taking blowups.

\begin{lemma}\label{lem:lift}
Let $s$ be a positive integer and let $P$ be a graph with
$\Delta(P)\le s$. If $P$ is a minor of $G$, then $B_s(G)$ contains a
subdivision of $P$.
\end{lemma}

\begin{proof}
Choose a minor model $\{W_v:v\in V(P)\}$ of $P$ in $G$. For every edge
$e=uv\in E(P)$, fix an edge of $G$ between $W_u$ and $W_v$, and write it
as
\[
x_{u,e}x_{v,e},
\qquad
x_{u,e}\in W_u,
\quad
x_{v,e}\in W_v.
\]
For each $v\in V(P)$, choose a root $z_v\in W_v$ and an injection
\[
\iota_v:
\{e\in E(P):e\text{ is incident with }v\}
\longrightarrow [s].
\]
For each edge $e$ incident with $v$, choose a simple path in $G[W_v]$ from
$z_v$ to $x_{v,e}$, say
\[
R_{v,e}=z_v=x_0,x_1,\ldots,x_k=x_{v,e}.
\]
Lift this path to $B_s(G)$ as
\[
(z_v,1),(x_1,\iota_v(e)),\ldots,(x_k,\iota_v(e));
\]
when $k=0$, the lifted path consists only of $(z_v,1)$. The lift is a
path because consecutive clusters are joined by complete bipartite graphs.

For a fixed $v$, the lifted paths corresponding to distinct incident edges
meet only at $(z_v,1)$: away from the root they use distinct labels, and
the underlying paths are simple. Lifted paths belonging to distinct
vertices of $P$ are disjoint because the branch sets $W_v$ are disjoint.

For each edge $e=uv\in E(P)$, join the terminal vertex of the lifted path
for $(u,e)$ to the terminal vertex of the lifted path for $(v,e)$. This
edge exists because $x_{u,e}x_{v,e}\in E(G)$ and the corresponding
clusters form a complete bipartite graph in $G$. Concatenating the two lifted paths with
this joining edge gives the path that replaces $e$.

Let $J$ be the subgraph consisting of all root vertices $(z_v,1)$, all
lifted paths, and all joining edges constructed above. In particular, if
$v$ is isolated in $P$, then $(z_v,1)$ is retained as an isolated branch
vertex of $J$. The replacing paths are internally vertex-disjoint and meet
precisely at their prescribed branch vertices. Hence $J$ is a subdivision
of $P$ in $B_s(G)$.
\end{proof}

We proceed now with the proof of our main theorem. 

\begin{proof}[\textbf{Proof of Theorem~\ref{thm:main}}]
Fix integers $d \ge 1$ and $s\ge2$. Apply Theorem~\ref{thm:KN} with $\alpha=1/(10s)$, and let
$\xi>0$ and $q_0,t_0\in\mathbb{N}$ be the resulting constants. Take
\[
T=\max\{100,t_0\}
\]
and put
\[
\eta=\frac{1}{10s\sqrt{T}}.
\]
Choose an integer $r\ge3$ sufficiently large such that
\begin{equation}\label{eq:r-choice}
2\sqrt{r-1}+1\le\eta r.
\end{equation}
Define
\[
D=sdr.
\]
Notice that this number depends only on $d$ and $s$.

Let $P$ be any graph with $\Delta(P)\le d$, and set
\[
m_P=\max\{|V(P)|,|E(P)|\}.
\]
Thus $P$ is among the graphs covered by $m_P$-minor-universality.

Choose an integer $N$ sufficiently large, subject to the parity condition
$rN$ even, so that Proposition~\ref{prop:spectral-existence} applies and
\begin{equation}\label{eq:N-choice}
N\ge 6s,
\qquad
\frac{N}{3s}\ge q_0,
\qquad\text{and}\qquad
\xi\frac{N\log T}{3s\log N}\ge m_P.
\end{equation}
Such admissible integers $N$ are unbounded, and $N/\log N\to\infty$ along
them. Proposition~\ref{prop:spectral-existence}, with $\eps=1$, therefore
gives an $N$-vertex $r$-regular graph $G$ such that
\[
\lambda(G)\le2\sqrt{r-1}+1\le\eta r.
\]

Let
\[
H=B_{sd}(G).
\]
Every vertex of $H$ has $sd$ neighbors in each of the $r$ clusters
corresponding to its neighbors in $G$. Hence
\[
\Delta(H)=sdr=D.
\]

Consider an arbitrary $s$-coloring of $V(H)$. For each $x\in V(G)$, assign to $x$ a majority color of its cluster $C_x$, breaking ties arbitrarily. At least $N/s$ vertices of $G$ receive the same assigned color. Without loss of generality, that color is red; let $S$ be the set of vertices in $V(G)$ assigned red. Thus $|S|\ge N/s$, and every cluster $C_x$ with $x\in S$ contains at least $d$ red vertices.

By Lemma~\ref{lem:robust}, there is a set $U\subseteq S$ with
$q:=|U|>N/(3s)$ such that
\[
Q:=G[U]
\]
is a $(1/(10s),T)$-expander. Since $N\ge6s$, we have $q>N/(3s)\ge2$, so all
logarithms below are positive. Moreover, $q\le N$ and $q>N/(3s)$. Hence,
\[
\xi\frac{q\log T}{\log q}
 \ge \xi\frac{N\log T}{3s\log N}
 \ge m_P.
\]
Also $q\ge q_0$ and $T\ge t_0$. Theorem~\ref{thm:KN} therefore implies
that $Q$ contains $P$ as a minor.

For each $x\in U$, choose exactly $d$ red vertices from the cluster
$C_x$. The subgraph induced by all the chosen vertices is a red copy of
$B_d(Q)$. Since $\Delta(P)\le d$ and $P$ is a minor of $Q$,
Lemma~\ref{lem:lift} gives a subdivision of $P$ inside this red copy of
$B_d(Q)$. Thus all vertices of the subdivision are red and this finishes the proof of the theorem.
\end{proof}

%--------------------------------------------------------------------
\section{Vertex Ramsey Theory for Trees and Bounded-Treewidth Graphs} \label{sec:treewidth}
%--------------------------------------------------------------------

In this section we prove that the families of trees, closed blowups of trees and graphs with bounded treewidth are vertex-$R_\Delta$-bounded.

We begin by recording the tree-embedding theorem of Friedman and
Pippenger, which is the main tool that we will use.

\begin{theorem}[Friedman--Pippenger~{\cite{FriedmanPippenger1987}}]
\label{thm:FP}
Let $d \ge 1$ and $n \ge 1$. If $H$ is a non-empty graph satisfying
\[
  |N_H(X)| \ge (d+1)|X|
\]
for every $X \subseteq V(H)$ with $|X| \le 2n-2$, then $H$ contains
every tree with at most $n$ vertices and maximum degree at most~$d$.
\end{theorem}

\begin{theorem}\label{thm:trees-vertex-ramsey}
The family of trees is vertex-$R_\Delta$-bounded. More precisely, for every integers $d \ge 1$ and $s\ge2$ there is an integer $D$, depending only on $d$ and $s$, such that for every tree $T$ with $\Delta(T) \le d$ there is a graph $H$ with $\Delta(H) \le D$ in which every $s$-coloring of $V(H)$ contains a monochromatic copy of~$T$.
\end{theorem}

\begin{proof}
Fix integers $d \ge 1$ and $s\ge2$. Set $T_0 = \max(d+1, 100)$ and put
\[
  \eta = \frac{1}{10s\sqrt{T_0}}.
\]
Choose an integer $r \ge 3$ large enough such that $2\sqrt{r-1}+1 \le \eta r$,
and set $D = r$. Notice that $D$ depends only on $d$ and $s$.

Let $T$ be any tree with $\Delta(T) \le d$ and $n = |V(T)|>1$ (if $n=1$ the result is trivial). Take $N \ge 30s^2(2n-2)T_0$ sufficiently large with $rN$ even. By Proposition~\ref{prop:spectral-existence} with $\eps = 1$, there exists an $N$-vertex $r$-regular graph $G$ with
\[
  \lambda(G) \le 2\sqrt{r-1}+1 \le \eta r.
\]

Consider an arbitrary $s$-coloring of $V(G)$. At least $N/s$ vertices receive the same color; without loss of generality that color is red. Let $S$ be the set of red vertices, so $|S| \ge N/s$. By Lemma~\ref{lem:robust}, there is a set $U \subseteq S$ with $q := |U| > N/(3s)$ such that $G[U]$ is a $(1/(10s), T_0)$-expander.

We verify the hypothesis of Theorem~\ref{thm:FP}. Since
\[
  q > \frac{N}{3s} \ge 10s(2n-2)T_0,
\]
we have $q/(10s\,T_0) > 2n-2$. Hence for every $X \subseteq V(G[U])$ with $|X| \le 2n-2$, the $(1/(10s),T_0)$-expansion condition applies and gives
\[
  |N_{G[U]}(X)| \ge T_0|X| \ge (d+1)|X|.
\]
By Theorem~\ref{thm:FP}, $G[U]$ contains $T$ as a subgraph. Since every
vertex of $G[U]$ belongs to $S$, this is a monochromatic copy of~$T$.
\end{proof}

Using the majority-coloring strategy from the proof of
Theorem~\ref{thm:main}, we extend this result to closed blowups of trees.

\begin{corollary}\label{cor:blowups-vertex-ramsey}
The family of closed blowups of trees is vertex-$R_\Delta$-bounded. More precisely, for every integers $d \ge 1$, $s\ge2$ and $k \ge 1$ there is an integer $D$, depending only on $d$, $s$ and $k$, such that for every tree $T$ with $\Delta(T) \le d$ there is a graph $H$ with $\Delta(H) \le D$ in which every $s$-coloring of $V(H)$ contains a monochromatic copy of the closed $k$-blowup~$T[k]$.
\end{corollary}

\begin{proof}
Fix integers $d,k \ge 1$ and $s\ge2$. Let $r$, $T_0$, and $\eta$ be the constants from the proof of Theorem~\ref{thm:trees-vertex-ramsey}, and set
\[
  D = sk(r+1)-1.
\]
This depends only on $d,k$ and $s$.

Let $T$ be any tree with $\Delta(T) \le d$ and $n = |V(T)|$. Choose $N \ge 30s^2(2n-2)T_0$ with $rN$ even, and let $G$ be an $N$-vertex $r$-regular graph with $\lambda(G) \le \eta r$, as given by Proposition~\ref{prop:spectral-existence}. Set $H = G[sk]$, the closed $sk$-blowup of~$G$. Each vertex of $H$ lies in a clique $C_x$ of size $sk$ and is adjacent to all $sk$ vertices in each of the $r$ neighboring clusters, so
\[
  \Delta(H) = (sk-1) + skr = sk(r+1)-1 = D.
\]

Consider an arbitrary  $s$-coloring of $V(H)$. For each $x \in
V(G)$, assign to $x$ the majority color in its cluster $C_x$, breaking
ties arbitrarily. At least $N/s$ vertices of $G$ receive the same majority color; without loss of generality that color is red. Let $S\subseteq V(G)$ be the set of majority-red vertices. Every cluster $C_x$ with $x \in S$ contains at least $k$ red vertices.

By Lemma~\ref{lem:robust} and the argument of Theorem~\ref{thm:trees-vertex-ramsey}, there is a set $U \subseteq S$ with $|U| > N/(3s)$ such that $G[U]$ is a $(1/(10s),T_0)$-expander, and Theorem~\ref{thm:FP} yields an embedding $\phi \colon V(T) \to U$ of $T$ in $G[U]$.

For each $v \in V(T)$, choose exactly $k$ red vertices from
$C_{\phi(v)}$; call this set $R_v$. The sets $\{R_v : v \in V(T)\}$ are
pairwise disjoint since $\phi$ is injective. Since $C_{\phi(v)}$ is a
clique of size $sk$ in $H$, the set $R_v$ induces a copy of $K_k$.
For each edge $vw \in E(T)$, the edge $\phi(v)\phi(w)$ lies in $E(G)$,
so all edges between $C_{\phi(v)}$ and $C_{\phi(w)}$ are present in $H$;
in particular, $R_v$ and $R_w$ span a copy of $K_{k,k}$. Therefore the
subgraph of $H$ induced by $\bigcup_{v \in V(T)} R_v$, together with all
edges between $R_v$ and $R_w$ for $vw \in E(T)$, is a monochromatic red
copy of $T[k]$.
\end{proof}

Recall that every graph of bounded treewidth and bounded maximum degree is
a subgraph of the closed blowup of a bounded-degree
tree (Lemma~\ref{lem:bddtreewidth}). Together with Corollary~\ref{cor:blowups-vertex-ramsey}, this gives the following.

\begin{corollary}\label{cor:treewidth-vertex-ramsey}
The family of graphs with bounded treewidth is vertex-$R_\Delta$-bounded.
More precisely, for every integers $s \ge 2$,  $d \ge 1$ and $w \ge 0$ there exists an integer $D$, depending only on $s, d$ and $w$, such that for every graph $P$ with maximum degree at most $d$ and treewidth at most $w$, there is a graph $H$ with maximum degree at most $D$ such that every $s$-coloring of $V(H)$ yields a monochromatic copy of $P$.
\end{corollary}

%--------------------------------------------------------------------
\section{Concluding remarks}
%--------------------------------------------------------------------
One can consider the vertex Ramsey  problem with $\rho(G)=\chi(G)$, where $\chi(G)$ is the chromatic number of $G$. Perhaps surprisingly, this turns out to be rather straightforward as we observe below.
\begin{proposition} \label{prop:chi}
Let $K_{c;n}$ be the complete $c$-partite graph with parts of size $n$. Then
\[
R_\chi^v(K_{c;n}) = 2c - 1.
\]
\end{proposition}

\begin{proof}
If $c=1$ the result is trivial. Let then $c\ge2$. For the lower bound, suppose that $H$ is a graph with $\chi(H) \le 2c - 2$. Consider an optimal coloring of $H$ with at most $2c-2$ colors and color half of the color classes red and the other half blue. Since each color class induces a graph with chromatic number at most $c-1$, there is no monochromatic copy of $K_{c;n}$. This shows $ R_\chi^v(K_{c;n})\ge 2c-1$.
To see $R_\chi^v(K_{c;n}) \leq 2c - 1$, observe that $K_{2c-1;2n} \xrightarrow{2}_{vtx} K_{c;n}$ 
as the pigeonhole principle implies that in any 2-coloring of $V(K_{2c-1;2n})$, one of the colors has at least  $n$ vertices in at least $c$ classes.\end{proof}

The proof of the general version with $s$ colors is analogous and gives us the following result.

\begin{proposition} \label{prop:chis}
Let $K_{c;n}$ be the complete $c$-partite graph with parts of size $n$. Then
\[
R_\chi^v(K_{c;n};s) = s(c-1)+1.
\]
\end{proposition}

If $G$ is any $n$-vertex graph with $\chi(G)=c$, then $G \subseteq K_{c;n}$ and this implies $R_\chi^v(G;s) \le s(c-1)+1$, by Proposition~\ref{prop:chis}. On the other hand, if $\chi(H)\le s(c-1)$ then its vertices can be partitioned into sets of size $s$, each inducing a graph of chromatic number at most $c-1$, so none can contain $G$. Thus, $R_\chi^v(G;s) = s(c-1)+1$.
Moreover, this generalizes naturally to $k$-uniform hypergraphs. Indeed, replacing the 
underlying graph structure with a $k$-uniform hypergraph $\mathcal{H}$, an analogous argument via the pigeonhole principle 
yields $R_\chi^v(K_{c;n}^{(k)};s) = s(c-1)+1$, where $K_{c;n}^{(k)}$ denotes the complete 
$c$-partite $k$-uniform hypergraph with parts of size $n$.

 The $n$ by $n$ square grid is the graph $P_n \square P_n$. One barrier to addressing Question~\ref{qn:vertex-main} is that we could not determine the answer for the grid graph. 
We end with this question.
\begin{question}
    Is the family of square grids vertex-$R_\Delta$-bounded?
\end{question}

\section*{Declaration of AI Assistance} ChatGPT pro was used to assist in producing the technical details in the proof of Lemma~\ref{lem:robust}. All other proofs in the paper, including the idea to use blowups of expanders and the Krivelevich-Nenadov result (Theorem~\ref{thm:KN}) to obtain minors  and then to apply Lemma~\ref{lem:lift} to obtain subdivisions  were due to the authors.
The authors  take full responsibility for the content of the paper, including all mathematical statements, proofs, and references.
\bibliographystyle{plain}
\bibliography{references}

\end{document}